\documentclass[10p,reqno,a4paper]{amsart}
\usepackage{amsmath,amssymb,amsthm}
\usepackage{url}
\usepackage{pgf,tikz}
\usepackage{dsfont}
\usetikzlibrary{arrows}
\usepackage{geometry}
\usepackage{fancyhdr}
\usepackage {array}
\usepackage{tabularx}
\usepackage{hyperref}
\hypersetup{colorlinks=true, linkcolor=blue, urlcolor=blue}

\title{Cycles of given lengths  in unicyclic components in sparse  random graphs}
\author{Marc Noy}
\address{Department of Mathematics, Universitat Polit\`ecnica de Catalunya\\ and Institut de Matem\`atiques de la UPC-BarcelonaTech (IMTech), Spain. }
\email[M.~Noy]{marc.noy@upc.edu}
\author{Vonjy Rasendrahasina}
\address{
Ecole Normale Sup\'erieure, Universit\'e d'Antananarivo,  Madagascar.}
\email[V.~Rasendrahasina]{rasendrahasina@gmail.com}
\author{Vlady Ravelomanana}
\address{
IRIF, Universit\'e de Paris, France.}
\email[V.~Ravelomanana]{vlad@irif.fr}
\author{Juanjo Ru\'e}
\address{
Department of Mathematics, Universitat Polit\`ecnica de Catalunya\\ and Barcelona Graduate School of Mathematics (BgsMath), Spain.}
\email[J.~Ru\'e]{juan.jose.rue@upc.edu}

\chardef\bslash=`\\ 

\newtheorem{thm}{Theorem}[section]

\newtheorem{lem}[thm]{Lemma}

\theoremstyle{definition}

\theoremstyle{remark}

\def \l{\left}
\def \r{\right}
\def \leq{\leqslant}
\def \geq{\geqslant}

\def\XNL{X_{n,\,M}^{(L)}}
\def\LAMBDA{\lambda_{L}}

\def\Esp{{\mathbb{E}}}
\newcommand{\Poisson}{\mathbin{\mathop{\mathrm{Poisson}}}}
\newcommand{\distrib}{\mathbin{\stackrel{\mathop{d}}{\longrightarrow}}}

\begin{document}
\maketitle

\begin{abstract}Let $L$ be subset of $\{3,4,\dots\}$ and let  $X_{n,M}^{(L)}$ be the number of cycles belonging to unicyclic components whose length is in $L$ in the random graph $G(n,M)$.
	We find the limiting distribution of $X_{n,M}^{(L)}$  in the subcritical regime $M=cn$ with $c<1/2$ and the critical regime $M=\frac{n}{2}\left(1+\mu n^{-1/3}\right)$ with $\mu=O(1)$. Depending on the regime and a condition involving the series
$\sum_{\ell \in L} z^\ell/(2\ell)$, we obtain in the limit  either a Poisson or a normal distribution as $n\to\infty$.\\

\end{abstract}

\renewcommand{\sectionmark}[1]{}

\section{Introduction}

A graph is unicyclic if it is connected and has a unique cycle.
We say that a cycle in a graph is \emph{isolated} if it is the unique cycle in a unicyclic connected component.
Let $G(n,M)$ be the  random graph with  $n$ vertices and exactly  $M$ edges drawn uniformly at random from the set of ${n \choose 2}$ possible edges.  This is the model introduced in  the seminal paper of Erd\H{o}s and R\'enyi~\cite{Erdos-Renyi}, in which each graph has the same probability
$$\binom{\binom{n}{2}}{M}^{-1}.$$

We are interested in the number of isolated cycles in $G(n,M)$  whose lengths are restricted to take certain values. More precisely, let  $\mathbb{N}_{\geq 3} = \{3,4,\dots\}$ and  $L$  a subset of $\mathbb{N}_{\geq 3}$. We denote by  $X_{n,M}^{(L)}$ the random variable equal to the number of isolated cycles in $G(n,M)$ whose lengths lie in  $L$. Our main result gives the limiting distribution of  $X_{n,M}^{(L)}$  for various values of $M$, corresponding to the so-called subcritical and critical regimes. Depending on the regime and a condition involving the generating function $\lambda(z) = \sum_{\ell \in L} z^\ell/(2\ell)$, we obtain in the limit as $n\to\infty$ either a Poisson or a normal distribution.

The number of cycles in $G(n,M)$ has been studied since the appearance of~\cite{Erdos-Renyi}. When $M = cn$,
 Erd\H{o}s and R\'enyi showed \cite[Theorem 3b]{Erdos-Renyi} that the number of  cycles of length $k$ converges to a Poisson law with parameter $(2c)^k/(2k)$. Let $X_{n,\, M}$ be the random variable equal to the number of isolated cycles in $G(n,M)$. When $M=cn$ and   $c<1/2$, asymptotically almost surely (that is, with probability tending to 1 as $n\to\infty$) all cycles are isolated. As a consequence we have
$$\lim_{n\to\infty} \Esp[X_{n,cn}] =\sum_{k\geq3} \frac{(2c)^k}{2k} =  \frac{1}{2}\log\frac{1}{1-2c} -c-c^2. $$

We next  recall the different regimes for sparse random graphs (see for instance \cite{JLR2000, AlonSpencer}). The following results hold asymptotically almost surely (shortened to a.a.s.).

\begin{itemize}
	\item \textbf{Subcritical regime}. When  $M=cn$  with $c<1/2$, the connected components of $G(n,M)$ are either trees or unicyclic graphs.
	
		\item \textbf{Barely subcritical regime}. When  $M=\frac{n}{2}\left(1-\mu n^{-1/3}\right)$ with $\mu\to\infty$ and $\mu=  o\left(n^{1/3}\right)$,
	%
	\item \textbf{Critical regime}. This is when $M=\frac{n}{2}\left(1+\mu n^{-1/3}\right)$ and $\mu=O(1)$. In this regime the connected components of  $G(n,M)$ are trees, unicyclic graphs, and \emph{complex} components. A complex component is obtained from a connected cubic multigraph $K$ by performing the following operations: first replace edges in $K$ by induced paths of any length so that to  obtain a simple graph  $C$, and then
	 attach rooted trees to the vertices of $C$.
	
	 	\item \textbf{Supercritical regime}. When  $M=cn$  with $c>1/2$, there exists a unique component $L$ of linear size and the remaining components are either trees or unicyclic graphs. The `Symmetry principle' (see \cite[Section 5.6]{JLR2000}) says that in this case $G(n.M) \backslash L$ in some sense `looks like' a subcritical random graph with suitable parameters.
\end{itemize}

\medskip
In the barely subcritical regime
Kolchin showed that if $r_0=\frac{1}{6}\log n -\frac{1}{2}\log \mu$, then the normalized random variable
$(X_{n,M}-r_0)/\sqrt{r_0}$ tends in distribution to a Gaussian law
(see \cite[Theorem 1.1.15]{Kolchin}). In the critical regime, Flajolet, Knuth and Pittel~\cite[Corollary~6]{FKP89}
 showed that $\Esp [X_{n,M}]\sim \frac{1}{6}\log n$.
By the so-called symmetry property  ~\cite[Theorem~5.24]{JLR2000}, $X_{n,M}$ properly normalized should also be Gaussian when $M=\frac{n}{2}(1+\mu n^{-1/3})$
and $\mu\to\infty$ with $\mu=o(n^{1/3})$.

Some results have been obtained fixing  a set $L$ of positive integers as possible cycle lengths.
Following~\cite{FKP89}, define an \emph{$L$-cycle}  as an isolated cycle whose length is in
$L$. Let $\XNL$ be  the number of
$L$-cycles in $G(n,\,M)$. It is shown in \cite[Corollary 7]{FKP89}  that  if
$\lim_{n\rightarrow\infty} \frac{2M}{n} = \lambda < 1$, then
the probability that a graph (or multigraph) with $n$ vertices and $M$ edges has no $L$-cycle  is equal to
\begin{equation}	\label{eq:FKP}
  \sqrt{1-\lambda} \, \exp{\l(\sum_{l\geq 1, \ell \notin L} \frac{\lambda^\ell}{2l} \r)} + O\left(n^{-1/2}\right)=\exp{\l(-\sum_{l\geq 1, \ell \in L} \frac{\lambda^\ell}{2l} \r)}+O\left(n^{-1/2}\right).
\end{equation}

Our results concern the   distribution of the random variables
 $X_{n,M}^{(L)}$.
In particular, we obtain full limiting distributions both in the subcritical and the critical regimes.

\begin{thm}\label{theo:nb_unicyclic}
 Let $L\subseteq \mathbb{N}_{\geq 3}$ and set  $\lambda_L(z)=\sum_{\ell\in L} \frac{z^\ell}{2\ell}$, considered as a function of one complex variable in the unit disk $|z| <1$.
Let $X_{n, M}^{(L)}$ be the random variable equal to the number of $L$-cycles in $G(n,M)$. Then the following holds:

\begin{enumerate}
\item[(A)] (Subcritical regime). Let $c=c(n)$ be such that $0<\limsup_{n\to\infty}c<1/2$ and $M=cn$. Then
\begin{equation} \label{eq:unicyclic_subcritical_poisson}
\frac{X_{n,\,M}^{(L)}}{\lambda_L(2c)} \distrib \Poisson \left(1 \right), \quad \hbox{ as $n\to\infty$}.
\end{equation}
\item[(B)] (Barely subcritical regime). Let $M=\frac{n}{2}(1-\mu n^{-1/3})$ with $\lim \mu=+\infty$ and $\mu=o(n^{1/3})$. Then two situations may happen: if $\lim_{n\to\infty} \lambda_L(\frac{2M}{n})<+\infty$, then
%
\begin{equation}\label{eq:Poisson-barely-subcritical}
\frac{X_{n,\,M}^{(L)}}{\lambda_L\left(\frac{2M}{n}\right)} \distrib \Poisson \left(1 \right), \quad \hbox{ as $n\to\infty$}.
\end{equation}

Otherwise, if $\lim_{n\to\infty} \lambda_L(\frac{2M}{n})=+\infty$, then
\begin{equation}\label{eq:unicyclic_subcritical}
\frac{X_{n,\,M}^{(L)}-\lambda_L\left(\frac{2M}{n}\right)}{\sqrt{\lambda_L\left(\frac{2M}{n}\right)}}
\distrib \mathcal{N}(0,1), \quad \hbox{ as $n\to\infty$}.
\end{equation}
\item[(C)] (Critical regime). Let $M=\frac{n}{2}(1+ \mu n^{-1/3})$, with $\mu=O(1)$. Let $\alpha$ be the unique positive solution of $\mu=\frac{1}{\alpha}-\alpha$. Then two situations may happen: if $\lim_{n\to\infty} \lambda_L(e^{-\alpha n^{-1/3}})<+\infty$, then
\begin{equation}\label{eq:Poisson-critical}
\frac{X_{n,\,M}^{(L)}}{\lambda_L\left(e^{-\alpha n^{-1/3}}\right)} \distrib \Poisson \left(1 \right), \quad \hbox{ as $n\to\infty$}.
\end{equation}
Otherwise, if $\lim_{n\to\infty} \lambda_L\left(e^{-\alpha n^{-1/3}}\right)=+\infty$, then
\begin{equation} \label{eq:unicyclic_critical}
\frac{X_{n,\,M}^{(L)}-\lambda_L\left(e^{-\alpha n^{-1/3}}\right)}{\sqrt{\lambda_L\left(e^{-\alpha n^{-1/3}}\right)}}
\distrib \mathcal{N}(0,1), \quad \hbox{ as $n\to\infty$}.
\end{equation}
\end{enumerate}
\end{thm}

Points (A), (B) and (C) in Theorem \ref{theo:nb_unicyclic}  are the contents of Theorems~\ref{theo:poisson}, \ref{lem:subritical_gaussian} and \ref{lem:critical_gaussian} given in the Section~\ref{PROOF:th:nb_unicyclic}.
We remark that in the previous statement there is \textit{no discontinuity} between  equations \eqref{eq:unicyclic_subcritical_poisson}--\eqref{eq:Poisson-barely-subcritical}--\eqref{eq:Poisson-critical}
and equations \eqref{eq:unicyclic_subcritical}--\eqref{eq:unicyclic_critical}:
the Taylor expansion of the term $e^{-\alpha n^{-1/3}}$  in the statement for the critical regime is equal to $1-\alpha n^{-1/3}+o(n^{-1/3})$, which coincides with the term $1-\mu n^{-1/3}$ in the barely subcritical region.

The proofs are based on estimating coefficients of generating functions by means of Cauchy integrals along suitable contours and applying the saddle-point method.

\noindent\textbf{Remarks.} Observe that \eqref{eq:FKP} follows directly from \eqref{eq:unicyclic_subcritical_poisson}. Let us mention that technical refinements of our techniques would provide similar results for the region just before the supercritical regime, namely $M=\frac{n}{2}(1+\mu n^{-1/3})$ when $\mu \to \infty$, $\mu =o(n^{1/12})$. We do not include the analysis of this region  because the computations become too involved.

Finally, one may wonder why in the previous theorem we do not have a corresponding result for the supercritical regime. The reason is that in this case our techniques, based on the detailed structure of $G(n,p)$ together with saddle-point estimates for the associated generating functions, do not apply in this situation. Given the Symmetry principle mentioned above, one should expect the number of $L$-cycles in the supercritical regime follows a limit Poisson law as in the subcritical regime, but the tools provided by the Symmetry principle do no seem precise enough to prove such a statement.

\section{Preliminaries and notation}

All graphs considered in this paper are labelled.
The \emph{size} of a graph is the number of vertices.
The \emph{excess} of a graph $G$ is the number of vertices minus the number of edges. In $G(n,M)$ the excess is $M-n$.

\subsection{Analytic combinatorics of graphs}
We use the language of analytic combinatorics as in  \cite{FSBook}. Given a generating function $A(x)=\sum_{n\geq 0} a_n x^n$, we write $[x^n]A(x)=a_n$. If $A(x)=\sum_{n\geq 0} a_n x^n$ and $B(x)=\sum_{n\geq 0} b_n x^n$, we write $A(x)\preceq B(x)$ if there exists $n_0$ such that $[x^n]A(x) \leq [x^n]B(x)$ for $n\geq n_0$. All the generating functions that appear in this work are \emph{exponential generating functions} of the form $\sum_{n\geq 0}a_nx^n/n!$, or EGF for short (see \cite[Chapter 2]{FSBook}).

We denote by $T(x)$ and $W_{-1}(x)$ the EGF of rooted and unrooted labelled trees, respectively. It is well known that
\begin{equation} \label{EGF_UNROOTED}
T(x) = x e^{T(x)} = \sum_{n=1}^{\infty} n^{n-1} \frac{x^n}{n!},\qquad W_{-1}(x) = T(x) - \frac{T(x)^2}{2}.
\end{equation}
The EGF $W_0(x)$ of unicyclic graphs (connected graphs with $n$ vertices and
$n$ edges) is given by (see for instance \cite[Equation~(3.5)]{JKLP93})
\begin{equation} \label{EGF_W0}
W_0(x) =\sum_{k\geq 3}\frac{T(x)^k}{2k} =-\frac{1}{2}\log{(1-T(x))} - \frac{T(x)}{2} - \frac{T(x)^2}{4}.
\end{equation}

We write $\lambda(t)=\sum_{k\geq 3}\frac{t^k}{2k}=-\frac{1}{2}\log(1-t)-{t}/{2}-{t^2}/{4}$, so that $W_0(x)=\lambda(T(x))$.

\subsection{From Poisson parametrizations to central limit theorems}
We include the following result by Kolchin that provides an approximation to a  normal law by a Poisson parametrization. 
\begin{thm}[$\textrm{\cite[Theorem~1.1.15]{Kolchin}}$] \label{theo:kolchin}
Let $k=\lambda_n+\rho_n\sqrt{\lambda_n}$. If $(1+\rho_n)^6/\lambda_n \to 0$ as $n\to\infty$ then
\[
e^{-\lambda_n} \frac{\lambda_n^k}{k!}=\frac{1}{\sqrt{2\pi  \lambda_n}} e^{-\rho_n^2/2} \l(1+\frac{\rho_n^3-\rho_n}{6\sqrt{\lambda_n}}
+O \l( \frac{1+\rho_n^6}{\lambda_n}\r) \r).
\]
\end{thm}

\section{Proof of Theorem~\ref{theo:nb_unicyclic}}\label{PROOF:th:nb_unicyclic}
We present separately the proof for each regime in Theorem~\ref{theo:nb_unicyclic}.
The main idea in all proofs is to encode the typical structure of random graphs in the regime under consideration using generating functions and then obtain large power estimates by means of saddle point bounds.

\subsection{Subcritical regime}
In this regime, the connected components of  $G(n,M)$ are a.a.s.\   a set of acyclic graphs (a forest) together with a set of unicyclic graphs.
We exploit this property in order to get the following result which refines the first statement in Theorem~\ref{theo:nb_unicyclic}:
\begin{thm} \label{theo:poisson}
Let $c$ such that $0<c<1/2$, and $M=cn$. Let $L\subseteq \mathbb{N}_{\geq 3}$  and $\lambda_L(z)=\sum_{\ell\in L} \frac{z^\ell}{2\ell}$.
Then the random variable $X_{n,\,M}^{(L)}$ equal to the  number of $L$-cycles satisfies
\begin{equation*}
\Pr\l[ X_{n,\,M}^{(L)}=k\r]= e^{-\lambda_L(2c)}\frac{\lambda_L(2c)^k}{k!}\left(1+O\left(n^{-1}\right)\right).
\end{equation*}
Moreover, if $k\to\infty$ as $n\to \infty$ then
\begin{equation*} \label{eq:upper_bound_large_r}
\Pr\l[ X_{n,\,M}^{(L)}=k\r] = O(k^{-k}).
\end{equation*}
\end{thm}

\begin{proof}

It suffices to consider graphs whose connected components are trees and unicyclic graphs.
Using the symbolic method we obtain that the probability that $G(n,M)$ contains exactly $k$ unicyclic components containing an $L$-cycle is equal to
\begin{equation} \label{def_proba_subcritical_cycle}
\Pr\l[ X_{n,\,M}^{(L)}=k\r]=\frac{n!}{\binom{\binom{n}{2}}{M}}[x^n]\frac{W_{-1}(x)^{n-M}}{(n-M)!}
 \frac{\lambda_L(T(x))^k}{k!}e^{W_0(x)-\lambda_L(T(x))}\, .
\end{equation}
The term $\frac{\lambda_L(T(x))^k}{k!}$ encodes the  components containing an $L$-cycle, while the term $e^{W_0(x)-\lambda_L(T(x))}$ encodes the rest of unicyclic components (whose lengths do not belong to $L$).
Using Cauchy integral's formula we get
\begin{eqnarray}\label{estimate1}
&[x^n]W_{-1}(x)^{n-M}\lambda_L(T(x))^k e^{W_0(x)-\lambda_L(T(x))}=\\
&\frac{2^{M-n}}{2\pi i} \oint \left(2W_{-1}(x)\right)^{n-M}\lambda_L(T(x))^k e^{W_0(x)-\lambda_L(T(x))} \frac{dx}{x^{n+1}} .\nonumber
\end{eqnarray}
After the change of variables $z=T(x)$, it becomes
\begin{equation} \label{integrale_cauchy}
[x^n]W_{-1}(x)^{n-M}\lambda_L(T(x))^k e^{W_0(x)-\lambda_L(T(x))} =\frac{2^{M-n}}{2\pi i} \oint g(z)\lambda_L(z)^k e^{nh(z)} \frac{dz}{z}\, ,
\end{equation}
where
\begin{eqnarray}
g(z)&=&(1-z)e^{\lambda(z)-\lambda_L(z)}, \label{eq:g}\\
h(z)&=&z-\log z+\l(1-\tfrac{M}{n} \r)\log\l(2z-z^2 \r) \label{fonction_h}.
\end{eqnarray}
Note that the function $h(z)$ given by \eqref{fonction_h} is exactly the same as \cite[Equation~(30)]{2XORSAT}, which satisfies the conditions $h'(2c)=h'(1)=0$.
In the range $M=cn$ with $0<c <\frac{1}{2}$, we can apply saddle-point methods by choosing a circular path $\{2ce^{i\theta}, \theta\in[-\pi,\pi)\}$ as the contour of integration.
As shown in~\cite{FKP89}, we split the integral in \eqref{integrale_cauchy} into three parts, namely $\int_{-\pi}^{-\theta_0}+\int_{-\theta_0}^{\theta_0}+\int_{\theta_0}^{\pi}$.
It suffices to integrate from $-\theta_0$ to $\theta_0$, for a convenient value of $\theta_0 $, because the remaining integrals can be bounded by the magnitude of the central integrand.
Following the proof of~\cite[Theorem~3.2]{2XORSAT}  and choosing $\theta_0=n^{-2/5}$ (so that $n\theta_0^2\to \infty$ but $n\theta_0^3\to 0$ as $n\to\infty$) we have
\begin{equation} \label{dev_expo_h1}
\exp\l( nh(2ce^{i\theta})\r)=\exp\l( nh(2c)-\tfrac{nc(1-2c)}{2(1-c)}\theta^2 \r)
\l(1+iO(n\theta^3)+O(n\theta^4) \r),
\end{equation}
and for all choices of $\theta$ in  $[-\pi,-\theta_0]\cup[\theta_0,\pi) $ we have
\begin{equation} \label{partie_omise}
\left|\exp\l( nh(2ce^{i\theta})-nh(2c) \r)\right| =  \exp\l(-O(n^{1/5}) \r).
\end{equation}
As $2c<1$ in the vicinity of $\theta_0$,  we have
\begin{equation} \label{developpement_g}
g\left(2ce^{i\theta}\right) = g(2c) \l(1 + i O(\theta)+O(\theta^2) \r),
\end{equation}
and
\begin{equation} \label{developpement_function_lambda}
\lambda_L(2ce^{i\theta})^k = \lambda_L(2c)^k \l(  1+iO(  \theta)+O( \theta^2) \r)
\end{equation}
for  fixed  $k\geq 0$. Using expansions \eqref{dev_expo_h1}, \eqref{developpement_g},
\eqref{developpement_function_lambda} and the bound \eqref{partie_omise}
we have
$$
\oint g(z)\lambda_L(z)^k e^{nh(z)} \frac{dz}{z}=
i\int_{-\theta_0}^{\theta_0}g(2ce^{i\theta}) \lambda_L(2ce^{i\theta})^k
e^{ nh(2ce^{i\theta})} d\theta\l(1+e^{-O(n^{1/5})}\r)
$$
$$
= ig(2c) \lambda_L(2c)^k e^{nh(2c)}\int_{-\theta_0}^{+\theta_0} e^{-n\sigma\tfrac{\theta^2}{2}}
\cdot (1+iO(\theta)+O(\theta^2)+iO(n\theta^3)+O(n\theta^4))d\theta \l(1+e^{-O(n^{1/5})}\r),
$$

\[
\begin{split}
\oint g(z)\lambda_L(z)^k e^{nh(z)} \frac{dz}{z}=
 i\int_{-\theta_0}^{\theta_0}g(2ce^{i\theta}) \lambda_L(2ce^{i\theta})^k
e^{ nh(2ce^{i\theta})} d\theta\l(1+e^{-O(n^{1/5})}\r) \\
= ig(2c) \lambda_L(2c)^k e^{nh(2c)}\int_{-\theta_0}^{+\theta_0} e^{-n\sigma\tfrac{\theta^2}{2}}
\cdot (1+iO(\theta)+O(\theta^2)+iO(n\theta^3)+O(n\theta^4))d\theta \l(1+e^{-O(n^{1/5})}\r),
\end{split}
\]
where $\sigma = \,{\frac {c \left( 1-2c \right) }{1-c}}$.
 If we set  $x =\sqrt{n \sigma} \theta$
the integral in the above equation becomes
\begin{equation}\label{eq:gaussian_integral}
\frac{1}{\sqrt{n\sigma}}
\int_{-\sigma^{1/2} n^{1/10}}^{\sigma^{1/2} n^{1/10}}e^{-\tfrac{x^2}{2}}
\l(1+iO\l(\tfrac{x}{\sqrt{n\sigma}}\r)+O \l( \tfrac{x^2}{n\sigma}\r)+
iO\l(n\tfrac{x^3}{\sqrt{n\sigma}^3}\r)+O\l(\tfrac{x^4}{n\sigma^2} \r) \r)dx.
\end{equation}
Observe that $\sigma=O(1)$ and the estimate \eqref{eq:gaussian_integral} is a real number (because \eqref{estimate1} is a real number).
Hence \eqref{eq:gaussian_integral} is equal to
\[
\frac{1}{\sqrt{\sigma n}} \int_{-\sigma^{1/2} n^{1/10}}^{\sigma^{1/2} n^{1/10}}e^{-\tfrac{x^2}{2}}
\l(1+O\l( \frac{x^2}{n}\r)+O\l(\frac{x^4}{n} \r) \r)dx.
\]
It follows  that
\[
\int_{-\theta_0}^{\theta_0}g(2ce^{i\theta}) \lambda_L(2ce^{i\theta})^k e^{nh(2ce^{i\theta})} d\theta=
\sqrt{\frac{2\pi}{\sigma n}} g(2c)\lambda_L(2c)^k e^{nh(2c)}
\l(1+O\l(n^{-1} \r)+ e^{-O(n^{1/5})}  \r).
\]
That is
\begin{equation} \label{resulta_col}
[x^n]W_{-1}(x)^{n-M}\lambda_L(T(x))^k e^{W_0(x)-\lambda_L(T(x))}= 2^{M-n} \frac{1}{\sqrt{2\pi\sigma n}} g(2c)\lambda_L(2c)^k e^{nh(2c)} \l(1+O\l(n^{-1} \r) \r)
\end{equation}
Using Stirling's formula for the corresponding  range of $M$, we have
\begin{equation} \label{stirling_facteur_multiplicative}
\frac{1}{\binom{\binom{n}{2}}{M}}\frac{n!}{(n-M)!k!}=\frac{1}{k!}
\sqrt{\frac{2\pi nM}{n-M}}\frac{2^Mn^nM^M}{n^{2M}(n-M)^{n-M}}
\exp\l(-2M+\frac{M}{n}+\frac{M^2}{n^2}\r)
 \l(1+O\l( n^{-1}\r) \r).
\end{equation}
Multiplying \eqref{resulta_col} and \eqref{stirling_facteur_multiplicative}, after  cancellations
 we obtain
\[
\Pr\l[ X_{n,\,M}^{(L)}=k\r]=e^{-\lambda_L(2c)}\frac{\lambda_L(2c)^k}{k!}\l(1+O\l(n^{-1} \r) \r).
\]
This proves the first part of the theorem.

Now, suppose that  $k\to \infty$ as $n\to\infty$. The previous arguments work in a similar way.
Instead of using the estimate \eqref{developpement_function_lambda}, which is only valid when $\theta$ is small enough, we exploit the fact that $\lambda_L$ has non-negative Taylor coefficients. Hence, Equation \eqref{developpement_function_lambda} can be replaced by the relation
\begin{equation*} \label{developpement_function_lambda2}
\left|\lambda_L(2ce^{i\theta})^k\right| \leq  \lambda_L(2c)^k,
\end{equation*}
which is valid for each choice of $\theta \in [-\pi,\pi)$. Applying the same arguments as before and that $\frac{1}{k!} < \frac{e^{k}}{k^k}$ for large $k$, we conclude that
\[
\Pr\l[ X_{n,\,M}^{(L)}=k\r]\leq e^{-\lambda_L(2c)}\frac{\lambda_L(2c)^k}{k!}\l(1+O\l(n^{-1} \r) \r)< C e^{-\lambda_L(2c)}\frac{(e\lambda_L(2c))^k}{k^k},
\]
for a suitable constant $C$. The second result in the theorem follows from the fact that $c<\frac{1}{2}$, and hence $\lambda_L(2c)$ is bounded.
\end{proof}

\subsection{Barely subcritical regime}

In the barely subcritical regime the asymptotic structure of $G(n,M)$ is the same as in the subcritical regime. However, the integration countour we use is slightly more complicated in order to encode cycles of arbitrary length.

\begin{thm} \label{lem:subritical_gaussian}
Let  $M=\frac{n}{2}(1-\mu n^{-1/3})$ with $\mu$ tending to infinity with $\mu=o\left(n^{1/3}\right)$. Let  $L\subseteq \mathbb{N}_{\geq 3}$ and  $\lambda_L(z)=\sum_{\ell\in L} \frac{z^\ell}{2\ell}$.
Then the random variable $X_{n,\,M}^{(L)}$ equal to  the number of $L$-cycles satisfies
\begin{equation}\label{thm.barelysub-poisson}
\mathrm{Pr}\left[X_{n,M}^{(L)}=k\right]=e^{-\lambda_L\l(\frac{2M}{n}\r)}\frac{\lambda_L\l(\frac{2M}{n}\r)^k}{k!}
\left(1+ O\left( \mu^{-3} \right)\right).
\end{equation}
Assume moreover that $\lim_{n} \lambda_L\left(\frac{2M}{n}\right)=\infty$. Then for fixed real numbers $y_0<y_1$
\begin{equation}\label{thm.barelysub-gaussian}
\Pr\left[y_0\leq \frac{X_{n,\,M}-\lambda_L\left(\frac{2M}{n}\right)}
{\sqrt{\lambda_L\left(\frac{2M}{n}\right)}}\leq y_1 \right]\to
\frac{1}{\sqrt{2\pi}} \int_{y_0}^{y_1} e^{-u^2/2}du, \quad \hbox{as $n \to \infty$}.
\end{equation}
\end{thm}

\begin{proof}
The arguments and notation are similar to the ones in the proof of  Theorem~\ref{theo:poisson}.
As mentioned in the proof of Theorem \ref{theo:poisson}, a.a.s.\ in this regime
$G(n,M)$ contains only trees and unicyclic graphs as components.
We need  estimates for (\ref{def_proba_subcritical_cycle}) in this new range of $M$. We use again the same methods as in the proof of~\cite[Theorem~3.2]{2XORSAT}.  Let
\begin{equation*} \label{theta_0}
\omega(n)=\frac{(n-2M)^{1/4}}{n^{1/6}},\,\qquad  \tau=\frac{n(n-M)}{M(n-2M)},\,\qquad  \theta_0=\sqrt{\frac{\tau}{n}}\omega(n).
\end{equation*}
Then $n\theta^2\to \infty$ and $n\theta^3\to 0$ as $n\to\infty$.
The expansion of  $h$ in the vicinity of $\theta_0$ is
\begin{equation} \label{dev-h-theta}
h\l(\frac{2M}{n}e^{i\theta}\r)=h\l(\frac{2M}{n}\r)-\frac{M(n-2M)}{2n(n-M)}\theta^2-
i\frac{(n^2-5nM+2M^2)M}{6(n-M)^2}\theta^3+O (\theta^4).
\end{equation}
For $ \theta\in[-\theta_0,+\theta_0]$,  $k =\Theta\l(\lambda_L(\frac{2M}{n})\r)$,
the expansion of $\lambda_L$ in the vicinity of $\theta_0$ is
\begin{equation} \label{dev_lambda_sous_critique_cas_2}
\begin{split}
\frac{\lambda_L\l(\frac{2M}{n}e^{i\theta}\r)^k} {\lambda_L\l(\frac{2M}{n}\r)^k}&=
1+iO\left( \frac{k}{\lambda_L\l(\frac{2M}{n}\r)}  \frac{n}{(n-2M)} \theta  \right)+
O\left( \frac{k^2}{\lambda_L\l(\frac{2M}{n}\r)^2} \frac{n^2}{(n-2M)^2}\theta^2\right) \\
&=1+iO\left( \frac{n}{(n-2M)} \theta  \right)+O\left(  \frac{n^2}{(n-2M)^2}\theta^2\right).
\end{split}
\end{equation}
The integrand can be bounded on $[-\pi,-\theta_0)\cup(\theta_0,\pi)$ because
\begin{equation} \label{partie_omise_sous_ctitique_2}
\l|\exp\l(nh\l(\frac{2M}{n}e^{i\theta}\r)-nh\l(\frac{2M}{n}\r)\r) \r|=O(e^{-\omega(n)^2/2}).
\end{equation}
Combining \eqref{dev-h-theta},
 \eqref{dev_lambda_sous_critique_cas_2} and
 \eqref{partie_omise_sous_ctitique_2}, we have
$$
\def\arraystretch{2}
\begin{array}{ll}
\mathrm{Pr}\left[X_{n,M}^{(L)}=k\right]=\frac{n!}{\binom{\binom{n}{2}}{M}(n-M)!}\frac{2^{M-n}}{2\pi } g\l(\frac{2M}{n}\r)
\exp\l(nh\l(\frac{2M}{n}\r)\r)\frac{\lambda_L\l(\frac{2M}{n}\r)^k}{k!} \times \\
\int_{-\theta_0}^{\theta_0}e^{-n\tau \frac{\theta^2}{2}}
\l(1+iO\left(  \frac{n}{(n-2M)} \theta  \right)+O\left( \frac{n^2}{(n-2M)^2}\theta^2\right)\r) \times \\
\l(1+in\frac{(n^2-5nM+2M^2)M}{6(n-M)^2}\theta^3+O (n\theta^4) \r)d\theta\l(1+O(e^{-\omega(n)^2/2})\r).
\end{array}
$$
We set  $\theta=\sqrt{{\tau}/{n}}x$ and the integral becomes
\[
\begin{split}
\sqrt{\frac{\tau}{n}}&\int_{-\omega(n)}^{\omega(n)}e^{-\frac{x^2}{2} }
\l(1+iO\l( \frac{n}{(n-2M)^{3/2}}x \r)+O\l(\frac{n^2}{(n-2M)^{3}} x^2\r)\r)\\
&\quad.\l(1+iO\l( \frac{n}{(n-2M)^{3/2}}x^3\r)+O\l(\frac{n}{(n-2M)^{3}}x^4 \r)\r)dx\\
&=\sqrt{\frac{\tau}{n}}\int_{-\omega(n)}^{\omega(n)}e^{-\frac{x^2}{2}}\l(1+O\l(\frac{n^2}{(n-2M)^3}x^4\r) \r)dx\\
&=\sqrt{\frac{2\pi \tau}{n}}\l(1+O\l(\frac{n^2}{(n-2M)^3}\r) \r).
\end{split}
\]
After simple algebraic manipulations as in the proof of  Theorem \ref{theo:poisson} we obtain
\[
\mathrm{Pr}\left[X_{n,M}^{(L)}=k\right]= e^{-\lambda_L\l(\frac{2M}{n}\r)}\frac{\lambda_L\l(\frac{2M}{n}\r)^k}{k!}
\left(1+ O\left( \mu^{-3} \right)\right) .
\]
This proves the first part of the theorem.

We assume now that $\lim_{n}\lambda_L\l(\frac{2M}{n}\r)=\infty$.
Set $k=\lambda_L\l(\frac{2M}{n}\r)+\rho_n \sqrt{\lambda_L\l(\frac{2M}{n}\r)}$
with $|\rho_n| =o\left(\lambda_L\l(\frac{2M}{n}\r)\right)^{1/6}$.
We can  apply Theorem~\ref{theo:kolchin} and obtain
\[
\begin{split}
  \Pr \l[ X_{n,\,M}^{(L)}=k\r]&=  \frac{1}{\sqrt{2\pi \lambda_L\l(\frac{2M}{n}\r)}}e^{-\rho_n^2/2}
\l(1+\frac{\rho_n^3-\rho_n}{\sqrt{\lambda_L\left(\frac{2M}{n}\right)}}+ O\l(\frac{1+\rho_n^3}{\sqrt{\lambda_L\l(\frac{2M}{n}\r)}} \r) \r)\, \\
&= \frac{1}{\sqrt{2\pi \lambda_L\l(\frac{2M}{n}\r)}}e^{-\rho_n^2/2}(1+o(1)).
\end{split}
\]
The central limit theorem for $X_{n,M}^{(L)}$ follows, that is,  for fixed real $y_0<y_1$ we have
\[
\Pr\left[y_0\leq \frac{X_{n,M}^{(L)}-\lambda_L\l(\frac{2M}{n}\r)}{\sqrt{\lambda_L\l(\frac{2M}{n}\r)}}\leq y_1 \right]\to
\frac{1}{\sqrt{2\pi}} \int_{y_0}^{y_1} e^{-u^2/2}du, \qquad \hbox{as $n\to\infty$}.
\]
\end{proof}

\subsection{Critical regime}
In this regime we have to take into account the appearance of complex components.
Let  $p_k(n,M;L,r)$ be the probability that $G(n,M)$ has a total excess
$r$ with exactly $k$ unicyclic components containing an $L$-cycle.
The following lemma gives an estimate for $p_k(n,M;L,r)$.

\begin{lem}\label{eq:unicyclic_critica}
Let $M=\frac{n}{2}(1+\mu n^{-1/3})$ with $\mu=O(1)$. Let $\alpha$ be the positive solution to  $\mu=\frac{1}{\alpha}-\alpha$. Let
 $$k=\lambda_L\l(e^{-\alpha n^{-1/3}}\r) +\rho\sqrt{\lambda_L\l(e^{-\alpha n^{-1/3}}\r) },
 $$
which satisfies $\rho=\omega\left(\lambda_L\l(e^{-\alpha n^{-1/3}}\r)^{1/6}\right)$.

 Then for fixed $r$ we have
\begin{equation*} \label{proba_critique_exces_fini}
\begin{split}
p_k(n,M; L,r)=e^{-\lambda_L\l(e^{-\alpha n^{-1/3}}\r)}\frac{\lambda_L\l(e^{-\alpha n^{-1/3}}\r)^k}{k!} \sqrt{2\pi}e_r\;A(3r+1/2,\mu)\cdot
\l(
1 +
O\l(n^{-1/12} \r)\r),
\end{split}
\end{equation*}
 where
\begin{equation*}
e_r = \frac{(6r)!}{ 2^{5r}3^{2r}(3r)!\,(2r)!},\,\,A(y,\mu)= \frac{e^{-\mu^3/6}}{3^{(y+1)/3}}\sum_{k\geq 0}
\frac{(\frac{1}{2}3^{2/3}\mu)^k}{k!\Gamma((y+1-2k)/3)}.
\end{equation*}
Moreover, for $r$ large enough there exist absolute constants $C>0$ and $\varepsilon>0$ such that
\begin{equation} \label{theo:excess_infini}
p_k(n,M;L,r)\leq e^{-\lambda_L\l(e^{-\alpha n^{-1/3}}\r)} \frac{\lambda_L\l(e^{-\alpha n^{-1/3}}\r)^k}{k!}  C  e^{-\varepsilon r}.
\end{equation}
\end{lem}

\begin{proof}
The proof is  based on analytic techniques introduced in~\cite{FKP89} and~\cite{JKLP93}; see also \cite{NRR15}.
The probability $p_k(n, M; L, r)$ is given by
\begin{equation} \label{def:proba_critical}
  p_k(n,M;L,r) = \frac{n!}{\binom{\binom{n}{2}}{M}}[x^n]
  \frac{W_{-1}(x)^{n-M+r}}{(n-M+r)!}E_r(x)
 \frac{\LAMBDA(T(x))^k}{k!}\, e^{W_0(x) - \LAMBDA(T(x))},
\end{equation}
where $E_r(x)$ is the EGF of complex components with total excess $r$ given by~\cite[Equation~(6.8)]{JKLP93}.
As shown in \cite{JKLP93}, when $r=o(n^{1/3})$, the series $E_r(x)$ can be approximated \cite[Equation~(6.8)]{JKLP93} by
$\frac{e_r}{(1-T(x))^{3r}}$, where
$$e_r=\frac{(6r)!}{ 2^{5r}3^{2r}(3r)!\,(2r)!},
$$
 and the error term is of order $ O\l(\frac{r^{3/2}}{n^{1/2}}\r)$.
In order to evaluate \eqref{def:proba_critical}  we have to compute the expression
\begin{equation} \label{eq:proba_critical}
  \frac{St(n,M,r)}{2\pi i} \oint
  (1-z)^{1-3r}e^{nh_1(z)}\, \frac{\LAMBDA(z)^k}{k!} \, e^{W_0(z) - \LAMBDA(z)}\, E_r(z) \, \frac{dz}{z}  ,
\end{equation}
where
\begin{eqnarray}
St(n,M,r)&=&\frac{n!}{\binom{\binom{n}{2}}{M}}\frac{2^{-n+M-r}e^ne_r}{(n-M+r)!}, \label{coeff_st}\\
h_1(z)&=&z-1-\log z +\l(1-\frac{M}{n}\r)\log(2z-z^2) \,\label{eq:h}.
\end{eqnarray}
We remark the difference between $h_1(z)$ and the function $h(z)$ defined in Equation \eqref{fonction_h}.
Note also that $h_1(z)$ is exactly the same as in \cite[Equation (10.12)]{JKLP93}, which
satisfies  $h_1(1) = h_1 ' (1) = 0$ and also  $h_1'' (1) = 0$  if $M = n/2$.
We now  follow the method of the
proof of \cite[Lemma~3]{JKLP93} in order to compute our integral
by choosing as path of integration
\begin{equation}\label{NEW_PATH}
  z=z(t)=e^{-\alpha n^{-1/3} - i t n^{-1/2}},
\end{equation}
where $\alpha$ is the unique positive solution of $ \mu = \frac{1}{\alpha}-\alpha$, and
$t$ belongs to the interval
$$\l( -\pi n^{1/4}\LAMBDA'(e^{-\alpha n^{-1/3}})^{-1/2},\pi n^{1/4}\LAMBDA'(e^{-\alpha n^{-1/3}})^{-1/2}\r).$$

Given that
\begin{equation}\label{eq:about_lambda_k}
  \begin{split}
 &  \frac{\LAMBDA(e^{-\alpha n^{-1/3} - i t n^{-1/2}})^k}{k!}
  e^{-\LAMBDA(e^{-\alpha n^{-1/3} - i t n^{-1/2}})}
  = \\
 &  \frac{\LAMBDA(e^{-\alpha n^{-1/3}})^k}{k!}
  e^{-\LAMBDA(e^{-\alpha n^{-1/3}})}  \qquad \times \\
 &  \left(
  1 + i O \left(  \frac{k}{\LAMBDA(e^{-\alpha n^{-1/3}})}
  \frac{\LAMBDA'\l(e^{-\alpha n^{-1/3}}\r)}{n^{1/2}}   t \right)
  + O \left(
  \frac{k^2}{\LAMBDA\l(e^{-\alpha n^{-1/3}}\r)^2}
  \frac{\LAMBDA'\l(e^{-\alpha n^{-1/3}}\r)^2}{n}    t^2   \right)
  \right)
  \end{split}
\end{equation}
as long as $k = O\l(\LAMBDA\l(e^{-\alpha n^{-1/3}}\r)\r)$, our choice ensures that the $O$ terms
 in~\eqref{eq:about_lambda_k} can be moved out of the integral.
By following the proof of  \cite[Equation~(10.1) of Lemma~3]{JKLP93}
we obtain that, for fixed values of $r$
\begin{equation}\label{eq:fixed_r}
p_k(n,M;L,r)=e^{-\LAMBDA(e^{-\alpha n^{-1/3}})}
\frac{\LAMBDA(e^{-\alpha n^{-1/3}})^k}{k!} \sqrt{2\pi }e_r A(3r+1/2,\mu)
\l(
1 +
O\l(n^{-1/12} \r)
+
O\l(\frac{\mu^4}{n^{1/3}} \r)
\r).
\end{equation}
Since  $k = O\l(\LAMBDA\l(e^{-\alpha n^{-1/3}}\r)\r)$ in the $O$ terms above we have
\[
\frac{\LAMBDA'\l(e^{-\alpha n^{-1/3}}\r)}{n^{1/2}} t
= O\l( \frac{\LAMBDA'\l(e^{-\alpha n^{-1/3}}\r)^{1/2}}{n^{1/4}} \r) =
O\l( \frac{1}{(1-e^{-\alpha n^{-1/3}})^{1/2}} \frac{1}{n^{1/4}} \r)
= O\l(n^{-1/12}\r).
\]
This proves the first statement of the theorem.

Next let us assume that $r\to\infty$.
We know that $E_r(z)\preceq \frac{e_r}{(1-T(z))^{3r}}$ (see for instance \cite[Lemma~4]{JKLP93}).
From \eqref{def:proba_critical} we have
\[
p_k(n,M;L,r)\leq\frac{n!}{\binom{\binom{n}{2}}{M}}[z^n]\frac{W_{-1}(x)^{n-M+r}}{(n-M+r)!}
 \frac{\LAMBDA(T(z))^k}{k!} \, e^{W_0(z) -\LAMBDA(T(z))} \frac{e_r}{(1-T(z))^{3r}}.
\]
Then, we obtain
\begin{equation}\label{eq:proba_upper_bound}
p_k(n,M;L,r) \leq
\frac{St(n,M,r)}{2\pi i}\oint \frac{z^r(2-z)^r}{(1-z)^{3r}} e^{nh_1(z)}
\frac{\LAMBDA(z)^k}{k!} e^{\lambda(z)-\LAMBDA(z)}
 \frac{dz}{z},
\end{equation}
with $h_1$ is defined by \eqref{eq:h}. In this case we take  as contour
of integration the circle $\{\delta e^{i\theta}:\theta\in[-\pi,\pi)\}$ with $\delta=1-\frac{r^{1/3}}{n^{1/3}}<1$.
On this circle, since $r\geq 1$,  for some constant $C$ and function $f(n)$ with $\lim_{n} f(n)=+\infty$, we have
\[
\begin{split}
  \frac{\LAMBDA(\delta)^k}{k!} e^{-\LAMBDA(\delta)} &
  \frac{\delta}{2\pi} \l(\frac{\delta(2-\delta)}{(1-\delta)^3}\r)^r
  e^{nh_1(\delta)}(1-\delta)^{1/2}
  \int_{-\pi}^{\pi}e^{-f(n)\theta^2}d\theta  < \frac{C}{\sqrt{n}} \, \delta^{r}
\l(\frac{\delta(2-\delta)}{(1-\delta)^3}\r)^r
  e^{nh_1(\delta)}.
\end{split}
\]
Note that $r\leq M =\frac{n}{2}(1+\mu n^{-1/3})$. Then  for $n$ large enough
\begin{equation}\label{eq:upper_bounb_large_power}
  \frac{\delta^r(2-\delta)^r}{(1-\delta)^{3r}} <\frac{n^r}{r^r},\,  nh_1(\delta)<\frac{13}{12}r+\frac{11}{6}\mu r^{2/3}.
\end{equation}
Using Stirling's formula we find that
\[
\frac{n!}{\binom{\binom{n}{2}}{M}(n-M+r)!}e^n 2^{-n+M-r} < \frac{n^{1/2}}{n^r}e^{-\mu^3/6+3/4}2^{-r},
\]
and for  $r\to\infty$, we have~
\begin{equation} \label{eq:asympt_e_r}
e_r=\frac{(6r)!}{ 2^{5r}3^{2r}(3r)!\,(2r)!} \leq  \frac{1}{r^{1/2}} \l(\frac{3r}{2e} \r)^r.
\end{equation}
Combining \eqref{eq:upper_bounb_large_power} and \eqref{eq:asympt_e_r}
in \eqref{eq:proba_upper_bound}, we deduce that~
\[
p_k(n,M;L,r)< \frac{c_0}{r^{1/2}}
\exp\l( -\frac{\mu^3}{6}+\frac{11}{6} \mu r^{2/3}
+\l(\frac{13}{12}+\log\frac{3}{4} -1 \r) r\r),
\]
for some constant $c_0>0$. Since $\frac{13}{12}+\log\frac{3}{4}-1 < 0$, when
  $r\to\infty$ we deduce
$$
p_k(n,M;L,r)\leq e^{-O(r)}.
$$
\end{proof}
We can now proof the main result in the critical regime.
\begin{thm} \label{lem:critical_gaussian}
Let  $M=\frac{n}{2}(1+\mu n^{-1/3})$ where $\mu=O(1)$. Let $\alpha$ be the positive solution to  $\mu=\frac{1}{\alpha}-\alpha$.
Let $L\subseteq \mathbb{N}_{\geq 3}$  and let $\lambda_L(z)=\sum_{k\in L} \frac{z^k}{2k}$.
Then the random variable $X_{n,\,M}^{(L)}$ equal $L$-cycles in   $G(n,M)$  satisfies
\begin{equation}\label{thm.barelysub-poisson}
\mathrm{Pr}\left[X_{n,M}^{(L)}=k\right]=e^{-\lambda_L\l(e^{-\alpha n^{-1/3}}\r)}\frac{\lambda_L\l(e^{-\alpha n^{-1/3}}\r)^k}{k!}
\l(1 +
O\l(n^{-1/12} \r)
\r).
\end{equation}
Moreover, assume that $\lim_{n\to\infty} \lambda_L\left(e^{-\alpha n^{-1/3}}\right)=+\infty$. Then, for each choice of real $y_0<y_1$
\begin{equation}\label{thm.barelysub-gaussian}
\Pr\left[y_0\leq \frac{X_{n,M}^{(L)}-\lambda_L\left(e^{-\alpha n^{-1/3}}\right)}
{\sqrt{\lambda_L\left(e^{-\alpha n^{-1/3}}\right)}}\leq y_1\right]\to
\frac{1}{\sqrt{2\pi}} \int_{y_0}^{y_1} e^{-u^2/2}du, \qquad \hbox{ when }n \to \infty.
\end{equation}
\end{thm}

\begin{proof}
By Lemma \ref{proba_critique_exces_fini} and the dominated convergence theorem,  $\Pr \l[ X_{n,\,M}^{(L)}=k\r] $ is equal to
\[
\begin{split}
\sum_{r\geq 0}p_k(n,M;L,r)
=\sum_{r\geq 0}e^{-\lambda_L\l(e^{-\alpha n^{-1/3}}\r)}\frac{\lambda_L\l(e^{-\alpha n^{-1/3}}\r)^k}{k!} \sqrt{2\pi}e_r\;A(3r+1/2,\mu)\cdot
\l(
1 +
O\l(n^{-1/12}\r)\r) .
\end{split}
\]
For  $\mu=O(1)$ Janson, Knuth, \L uczak and Pittel~\cite[Equation~(13.17) and Corollary p. 61]{JKLP93}
 have shown that the probability that $G(n,\,M)$ has total excess $r$ is asymptotically
$\sqrt{2\pi}e_rA(3r+1/2,\mu)$, and  that the $s$-th moment of the excess $r$ satisfies
$\sum_{r\geq 0}\sqrt{2\pi}e_r r^{s}A(3r+1/2,\mu)=O(\mu^{3s})=O(1)$. Hence $\sum_{r\geq 0}\sqrt{2\pi}e_r A(3r+1/2,\mu)=1$. This shows relation \eqref{thm.barelysub-poisson}.

In order to prove \eqref{thm.barelysub-gaussian}, we apply Theorem~\ref{theo:kolchin} by choosing $$k=\lambda_L\l(e^{-\alpha n^{-1/3}}\r)+\rho_n\sqrt{\lambda_L\l(e^{-\alpha n^{-1/3}}\r)},$$
with $\lambda_L\l(e^{-\alpha n^{-1/3}}\r)^{1/6}=o(\rho_n).$
\end{proof}


\section{Acknowledgments}
Part of this research was done when the J.R. was visiting IRIF --
 Universit\'e Paris Denis Diderot in July 2018. J.R. thanks the hospitality of the institution during his research stay. V.R. was partially supported by grants ANR 2010 BLAN 0204 (Magnum). V.R. and J.R. were supported by the Project PHC Procope ID-57134837
 `Analytic, probabilistic and geometric Methods for Random Constrained graphs'. J.R. was partially supported by the FP7-PEOPLE-2013-CIG project CountGraph (ref. 630749). Finally, M.N. and J.R. were supported by the Spanish project MTM2017-82166-P and by the H2020-MSCA-RISE-2020 project RandNET (ref. 101007705).

\bibliographystyle{plain} 
\bibliography{some_properties_critical_graphs2}

\end{document}